\documentclass[11pt]{amsart}

\usepackage{amsfonts, amssymb, amscd}

\usepackage{verbatim}
\usepackage{amssymb}
\usepackage{mathrsfs}
\usepackage{graphicx}
\usepackage[all]{xy}
\usepackage{tikz-cd}
\usepackage{subfiles}
\usepackage[toc,page]{appendix}
\usepackage{mathtools}
\usepackage{upgreek}

\usepackage{hyperref}

\newcommand{\Zz}{\mathbb{Z}}

\newcommand{\Pp}{\mathbb{P}}
\newcommand{\Rr}{\mathbb{R}}
\newcommand{\Qq}{\mathbb{Q}}
\newcommand{\Nn}{\mathbb{N}}

\newcommand{\pet}{\operatorname{pet}}
\newcommand{\PET}{\operatorname{PET}}

\newcommand{\Supp}{\operatorname{Supp}}

\newtheorem{theorem}{Theorem}[section]
\newtheorem{lemma}[theorem]{Lemma}
\newtheorem{proposition}[theorem]{Proposition}
\newtheorem{definition}[theorem]{Definition}

\newtheorem{corollary}[theorem]{Corollary}
\newtheorem{remark}[theorem]{Remark}
\newtheorem{conjecture}[theorem]{Conjecture}

\begin{document}

\title{On accumulation points of pseudo-effective thresholds}

\begin{abstract}
We characterize a $k$-th accumulation point of pseudo-effective thresholds of $n$-dimensional varieties as certain invariant associates to a numerically trivial pair of an $(n-k)$-dimensional variety. This characterization is applied towards Fujita's log spectrum conjecture for large $k$.
\end{abstract}

\author{Jingjun Han}
\address{Beijing International Center for Mathematical Research, Peking University,
Beijing 100871, China} \email{hanjingjun@pku.edu.cn}

\curraddr{Department of Mathematics, Johns Hopkins University,
Baltimore, MD 21218, USA}\email{jhan@math.jhu.edu}

\author{Zhan Li}
\address{Department of Mathematics, Southern University of Science and Technology, 1088 Xueyuan Rd, Shenzhen 518055, China} \email{lizhan@sustech.edu.cn}

\maketitle

\tableofcontents

\section{Introduction}\label{sec: introduction}

It has long been realized that the behavior of certain invariants, such as log canonical threshold and minimal log discrepancy, relates to deep results in birational geometry. A celebrated conjecture of Shokurov \cite{Sho88, Fli92} predicts that the log canonical thresholds satisfy the ascending chain condition (ACC). This conjecture has been extensively studied before it is fully established in \cite{HMX14}. Another problem is the distribution of log canonical thresholds. It is speculated that the accumulation points of $n$-dimensional log canonical thresholds should lie in the set of $(n-1)$-dimensional log canonical thresholds. This is established in \cite{HMX14} under certain restrictions on the coefficients of boundary divisors.

\medskip

Pseudo-effective threshold is another invariant of this kind. Roughly speaking, a pseudo-effective threshold is a measurement of how far a divisor is away from being effective with respect to a given divisor. Fujita defines the (log) Kodaira energy \cite{Fuj92, Fuj96} which is nothing but the negative of the corresponding pseudo-effective threshold. From the classification perspective, smooth varieties with large pseudo-effective thresholds have been extensively studied.

\medskip

Pseudo-effective threshold can be analogized to log canonical threshold in many aspects. Fujita proposed the spectrum conjecture on pseudo-effective threshold (\cite[(3.2)]{Fuj96}) which is an analogy for ACC conjecture on log canonical thresholds. One version of this conjecture predicts that the set of pseudo-effective thresholds is an ACC set. This conjecture has been confirmed in \cite{Fuj96} for $3$-folds, in \cite{DiC16, DiC17} for arbitrary varieties, and in \cite{HL17} for more general setting. Fujita's log spectrum conjecture (Fujita attributed this to Shokurov in \cite[(3.7)]{Fuj96}) can be viewed as an analogy to the aforementioned conjecture on the accumulation points of log canonical thresholds. Notice that the term ``log spectrum conjecture'' in \cite{DiC16, HL17} is used in a different context. 

\medskip

In the terminology of this paper, Fujita's log spectrum conjecture can be stated as the following (the original conjecture is stated in terms of the Kodaira energy). Let $(X, \Delta)$ be a log smooth variety with $\Delta$ a reduced divisor. Let $M$ be an ample Cartier divisor on $X$ and $\pet(X, \Delta; M) \in \Rr_{\geq 0}$ be the pseudo-effective threshold of $K_X+\Delta$ with respect to $M$ (see Section \ref{sec: accumulation points}). Let $\PET_n$ be the set of all such $\pet(X, \Delta; M)$ with $\dim X=n$. For a subset $S \subseteq \Rr$, let $\lim^1 S$ be the set of accumulation points of $S$, and let the set of the $k$-th accumulation points be $\lim^k S = \lim^1(\lim^{k-1} S)$ for any $k \in \Nn$.

\begin{conjecture}[{Fujita's log spectrum conjecture \cite[(3.7)]{Fuj96}}]\label{conj: Fujita's log spectrum conjecture} Under the above notation, $\lim^{k}(\PET_n) \leq n-k$ for any positive integer $k\leq n$ .
\end{conjecture}

This conjecture is widely open and it seems that no complete answer is known even for surfaces. The main result of this paper is to obtain a characterization of the $k$-th accumulation points under reasonable assumptions on the coefficients of the boundary divisors.

\medskip

Let $I \subseteq [0,1]$ be a subset, then $I$ is called a DCC set if it satisfies the descending chain condition. Let $I_+ \coloneqq \{\sum n_i a_i \leq 1 \mid a_i \in I, n_i \in \Nn\} \cup\{0\}$, and
\begin{equation}\label{eq: PET_n(I)}
\begin{split}
\PET_n(I) = \{ & \pet(X, \Delta; M) \mid (X, \Delta) \text{~is lc}, \text{~coefficients of }\Delta \text{~are in~} I, \\
&M \text{~is a nef and big Cartier divisor}, \dim X=n\}.
\end{split}
\end{equation}

\begin{theorem}\label{thm: accumulation point of pet}
Let $I$ be a DCC set such that $I=I_+$. Assume that $1\in I$ with $1$ the only possible accumulation point, then for any $1 \leq k \leq n$, $\lim^k (\PET_n(I)) \subseteq K_{n-k}(I)$.
\end{theorem}

The definition of the set $K_{n-k}(I)$ is given in Section \ref{subsec: a characterization} Equation \eqref{eq: sets}. Roughly speaking, $c \in K_{n-k}(I)$ if there exists an $(n-k)$-dimensional Fano variety $X$ of Picard number $1$, such that there exists a numerically trivial generalized klt pair $K_{X}+B+cM \equiv 0$ with either $M \not\equiv 0$ or $B$ has a coefficient $\frac{m-1+f+kc}{m}$ with $f \in I_+$ and $m, k \in \Nn$. Although the last condition sounds cumbersome, in view of Fujita's log spectrum conjecture \ref{conj: Fujita's log spectrum conjecture}, it is the most simple case. In fact, as $\frac{m-1+f+kc}{m}<1$, we must have $c<1$. Hence, an accumulation point of such $c$ automatically lies in the range predicted by the conjecture when $k \leq n-1$.

\medskip

The proof of Theorem \ref{thm: accumulation point of pet} relies heavily on the minimal model program, especially the resent progress on ACC for log canonical thresholds, generalized polarized pairs and BAB conjecture. Part of the argument follows along the same line as that on accumulation points of log canonical thresholds in \cite{HMX14}. The generalized polarized pairs have the advantage over the standard pairs in dealing with such problems. In fact, suppose that $c$ is the pseudo-effective threshold of $M$ with respect to $K_X+\Delta$, then $(X, \Delta+cM)$ may not necessarily be a standard part, but it is a generalized polarized pair. This point of view has been used in \cite{HL17} to deal with ACC for pseudo-effective thresholds.

\medskip

As an application, we give an upper bound for the $k$-th accumulation points when $k$ is large.

\begin{proposition}\label{prop: Fujita's conjecture for n-1}
Let $I \subseteq [0,1]$ be a DCC set. Suppose that $1$ is the only possible accumulation point of $I$. Then  $\lim^{n-1} (\PET_n(I)) \leq 2$ and $\lim^{n} (\PET_n(I)) \leq 1$.
\end{proposition}

Notice that under the above condition, such upper bounds are sharp. Hence at least for surfaces, we have upper bounds for all terms of pseudo-effective accumulation points.

\medskip

Finally, upper bounds for all the $k$-th accumulation points are obtained in \cite{Li18} by a modified version of Theorem \ref{thm: accumulation point of pet} together with other techniques. 

\medskip

The paper is organized as follows. In Section \ref{sec: preliminaries}, we recall the necessary background for the proof. In Section \ref{sec: accumulation points}, Theorem \ref{thm: accumulation point of pet} is proved after a series of reductions. An application towards Fujita's log spectrum conjecture is given at the end of this section.

\medskip

\noindent\textbf{Acknowledgements}.
J. H. thanks his advisor Gang Tian for constant supports and encouragement. Z. L. thanks Chen Jiang for extensive discussions on special cases related to the problem. The first part of the paper was completed in the summer of 2017 while Z. L. stayed at the University of L\"ubeck and he thanks the quiet environment provided by the Zentrale Hochschulbibliothek L\"ubeck. Both authors thank Chenyang Xu for constant supports. Both authors also thank the anonymous referee for carefully reading the manuscript and providing constructive suggestions. This work is partially supported by NSFC Grant No.11601015.

\section{Preliminaries}\label{sec: preliminaries}
We work with complex numbers. Throughout the paper, $\Zz$ denotes the set of integers and $\Nn$ denotes the set of positive integers. We collect relevant definitions and results on generalized polarized pairs developed in \cite{BZ16}.

\subsection{Generalized polarized pairs}
\begin{definition}[{Generalized polarized pair, \cite[Definition 1.4]{BZ16}}]
A generalized polarized pair consists of a normal variety $X'$ equipped with projective morphisms 
\[
X \xrightarrow{f}X' \to Z,
\] where $f$ is birational and $X$ is normal, an $\Rr$-boundary $B'$, and an $\Rr$-Cartier divisor $M$ on $X$ which is nef$/Z$ such that $K_{X'}+B'+M'$ is $\Rr$-Cartier, where $M'\coloneqq f_*M$. We call $B'$ the boundary part and $M$ the nef part.
\end{definition}

For simplicity, we sometimes also call $M'$ the nef part without referring to $M$. When $Z$ is a closed point, we write $X \xrightarrow{f}X'$ instead of $X \xrightarrow{f}X' \to \{\rm pt\}$. From the definition, we see that $X$ could be replaced by any log resolution over $X$, and $M$ could be replaced by the pullback of $M$ accordingly. To define the \emph{generalized log discrepancy} of a divisor $E$ over $X'$, considering a high enough model $X$ which contains $E$ (say a resolution as above), if
\[
K_X+B+M=f^*(K_{X'}+B'+M'),
\] then the generalized log discrepancy of $E$ is defined as $1-{\rm mult}_EB$ (see \cite[Definition 4.1]{BZ16}). We say that $(X',B'+M')$ is generalized lc (resp. generalized klt) if the generalized log discrepancy of any prime divisor is $\geq 0$ (resp. $>0$). Just as in the standard setting, one can define generalized non-klt/lc centers and generalized non-klt/lc places. Besides, as $M$ is a nef divisor, if $M'$ is $\Rr$-Cartier, then by the negativity lemma (see \cite[(1.1)]{Sho92} or \cite[Lemma 3.6.2]{BCHM10}), $f^*M' =M+E$ with $E \geq 0$ an exceptional divisor. In particular, this implies that if $K_{X'}+B'$ is $\Rr$-Cartier, then the log discrepancy of a divisor $F$ with respect to $(X', B')$ is no less than the generalized log discrepancy of $F$ with respect to $(X', B'+M')$.

\subsection{Generalized adjunction}\label{subsec: adjunction}
Generalized adjunction for generalized polarized pairs is defined in \cite[Definition 4.7]{BZ16}, it will be used in the induction argument to lower the dimensions.

\begin{definition}[Adjunction for generalized polarized pairs]\label{def: adjunction}
Let $(X',B'+M')$ be a generalized polarized pair with data $X \xrightarrow{f} X' \to Z$ and $M$. Assume that $S'$ is the normalization of a component of  $\lfloor B' \rfloor$ and $S$ is its birational transform on $X$. Replacing $X$ we may assume that $f$ is a log resolution of $(X',B'+M')$. Write 
\[
K_X +B+M=f^*(K_{X'} +B' +M'),
\] and let
\[
K_S +B_S +M_S \coloneqq (K_X +B+M)|_S
\]
where $B_S = (B-S)|_S$ and $M_S =M|_S$. Let $g$ be the induced morphism $S\to S'$ and let $B_{S'} = g_*B_S$ and $M_{S'} =g_*M_S$. Then we get the equality
\[
K_{S'}+B_{S'}+M_{S'} = (K_{X'}+B'+M')|_{S'}
\] which is refered as generalized adjunction.
\end{definition}

The morphisms $S \xrightarrow{g} S' \to Z$ and $M|_S$ give a generalized polarized structure on $(S', B_{S'}+M_{S'})$. The singularities and coefficients behave just as in the standard adjunctions (see \cite{Sho92, Fli92}).  To be precise, when $(X',B'+M')$ is generalized lc, $B_{S'}$ is a boundary divisor on $S'$ (see \cite[Remark 4.8]{BZ16}) and $(S', B_{S'}+M_{S'})$ is still generalized lc. Moreover, suppose that $M= \sum \mu_j M_j$ with $M_j$ nef/$Z$ Cartier divisor for each $j$, and $B'=\sum b_i B_i'$ is the prime decomposition. Assume that $M_j'$ are $\Qq$-Cartier divisors, then the coefficient of a divisor $V$ in $B_{S'}$ is either $1$ or of the form
\begin{equation}\label{eq: coefficients}
\frac{m-1}{m}+\sum b_i \frac{d_i}{m}+\sum \mu_j \frac{e_j}{m},
\end{equation} with $m, d_i, e_j \in \Nn$ (see \cite[proof of Proposition 4.9]{BZ16}). In fact, when the coefficient of $V$ is less than $1$, $m$ is the Cartier index along $V$ (see \cite[(16.6.3)]{Fli92}). The term $\frac{m-1}{m}+\sum b_i \frac{d_i}{m}$ in \eqref{eq: coefficients} comes from the adjunction $(K_{X'}+B')|_{S'}$. The $M'$ and $M$ contribute to the last term $\sum \mu_j \frac{e_j}{m}$. Indeed, for $M'_j = f_*M_j$, we have $f^*M'_j=M_j+E_j$ with $E_j \geq 0$. $mM'_j$ is Cartier along the image of $V$, and $M_j$ is Cartier. Let $\tilde V$ be the strict transform of $V$ in $S\subset X$. Then $mE_j= mf^*M'_j-mM_j$ is a Cartier divisor along $V$. Thus the coefficient of $\tilde V$ in $E_j|_{S}$ is of the form $\frac{e_j}{m}$ with $e_j \in\Nn$, and so is the coefficient of $V$ in $g_*(E_j|_{S})$.

\subsection{MMP for generalized polarized pairs}
Although the the minimal model program (MMP) for generalized polarized pairs is not established in the full generality, some of the most important cases could be derived from the standard MMP. The following results are contained in \cite[\S 4]{BZ16} which are elaborated in \cite[\S 3]{HL18}.

\medskip

Assume that $K_{X'}+B'+M'+A'$ is nef/$Z$ for some $\Rr$-Cartier divisor $A' \geq 0$ which is big/$Z$. Moreover, assume that
\begin{quote}
($\star$) for any $s \in (0,1)$ there exists a boundary 
\[
\Delta' \sim_\Rr B'+sA'+M'/Z
\] such that $(X',\Delta'+(1-s)A')$ is klt.
\end{quote}

Condition $(\star)$ will be satisfied if $A'$ is generalized ample/$Z$ and either

(i) $(X', B'+M')$ is generalized klt, or

(ii) $(X', B'+M')$ is generalized lc and $(X',0)$ is klt.

\medskip

Under the above assumptions, we can run a $(K_{X'}+B'+M')$-MMP with scaling of $A'$. Under suitable assumptions, this MMP terminates. This is summarized in the following lemma.

\begin{lemma}[{\cite[Lemma 4.4]{BZ16}}]\label{lem: MMP}
Let $(X', B' + M')$ be a $\Qq$-factorial generalized lc polarized pair with data $X \xrightarrow{f} X' \to Z$ and $M$. Assume that $(X', B' + M')$ satisfies (i) or (ii) above. Run an MMP/$Z$ on $(K_X'+ B' + M')$ with scaling of some general ample/$Z$ $\Rr$-Cartier divisor $A'\geq 0$. Then the following hold:
\begin{enumerate}
\item Assume that $K_X'+ B' + M'$ is not pseudo-effective/$Z$. Then the MMP terminates with a Mori fibre space.
\item Assume that
\begin{itemize}
\item $K_{X'} + B' + M'$ is pseudo-effective/$Z$,
\item $(X', B' + M')$ is generalized klt, and that
\item $K_{X'} +(1+\alpha)B' +(1+\beta)M$ is $\Rr$-Cartier and big/$Z$ for some $\alpha, \beta \geq 0$. 
\end{itemize}
Then the MMP terminates with a minimal model $X''$ and $K_{X''} + B'' + M''$ is semi-ample/$Z$, hence it defines a contraction $\phi: X'' \to T''/Z$.
\end{enumerate}
\end{lemma}

As an application, one can obtain an analogy for dlt modifications. However, we tacitly avoid to introduce the notation of generalized dlt pairs for simplicity.

\begin{proposition}[{\cite[Lemma 4.5]{BZ16}}]\label{prop: dlt}
Let $(X', B' + M')$ be a generalized lc polarized pair with data $X \xrightarrow{f} X' \to Z$ and $M$. Let $S_1, \ldots, S_r$ be prime divisors on birational models of $X'$ which are exceptional/$X'$ and whose generalized log discrepancies with respect to $(X', B' + M')$ are at most $1$. Then perhaps after replacing $f$ by a high resolution, there exist a $\Qq$-factorial generalized lc polarized pair $(X'', B'' + M'')$ with data $X \xrightarrow{g} X'' \to Z$ and $M$, and a projective birational morphism $\phi: X'' \to X'$ such that
\begin{enumerate}
\item $S_1, \ldots, S_r$ appear as divisors on $X''$,
\item each exceptional divisor of $\phi$ is one of the $S_i$ or is a component of $\lfloor B''\rfloor$, 
\item $(X'',0)$ has klt singularities,
\item $K_{X''}+B''+M'' = \phi^*(K_{X'}+B'+M')$, and
\item if $\lfloor B'' \rfloor = 0$, then $(X'', B''+M'')$ is generalized klt.
\end{enumerate}
In particular, the exceptional divisors of $\phi$ are exactly the $S_i$ if $(X', B' + M')$ is generalized klt.
\end{proposition}
\begin{proof}
Each claim except (3) and (5) is explicitly stated in \cite[Lemma 4.5]{BZ16}. (3) holds because we obtain $X''$ by running an MMP/$X'$ with scaling, and in each step of this process, we construct a klt pair. Hence, in each step, the variety itself is klt. In particular, in the last step, $(X'',0)$ is klt. For (5), when $\lfloor B'' \rfloor = 0$, by the construction of \cite[Lemma 4.5]{BZ16}, we have $\lfloor B' \rfloor = 0$, and the generalized log discrepancy of any exceptional divisor in the log resolution $X \to X'$ is larger than $0$ (otherwise, it would appear in $\lfloor B'' \rfloor$). The notion of generalized klt singularity is independent of the log resolution, hence the original pair $(X', B'+M')$ is generalized klt. This implies that $(X'', B''+M'')$ is also generalized klt.
\end{proof}

We can extract the divisors which are exactly $S_1, \ldots, S_r$ when $X'$ has better singularities.
 
\begin{proposition}[{\cite[Lemma 4.6]{BZ16}}]\label{lem: dlt for klt}
Under the notation and assumptions of Proposition \ref{prop: dlt}, further assume that $(X',C')$ is klt for some $C'$, and that the generalized log discrepancies of the $S_i$ with respect to $(X', B'+ M')$ are $< 1$. Then we can construct $\phi$ so that in addition it satisfies:
\begin{enumerate}
\item its exceptional divisors are exactly $S_1, \ldots, S_r$, and
\item if $r = 1$ and $X'$ is $\Qq$-factorial, then $\phi$ is an extremal contraction.
\end{enumerate}
\end{proposition}

\subsection{A collection of relevant results}
For reader's convenience, we collect some relevant results. A set of real numbers is called ACC (ascending chain condition) if there is no strictly increasing infinite sequence. In the same fashion, we define a set to be DCC (descending chain condition). The first two results are called ACC for generalized lc thresholds and global ACC (\cite[Theorem 1.5, 1.6]{BZ16}) respectively. They generalize the corresponding results in standard setting (\cite{HMX14}). For the definition of the generalized lc threshold, see \cite[Definition 4.3]{BZ16}.

\begin{theorem}[ACC for generalized lc thresholds]\label{thm: ACC for glct}  Let $\Lambda$ be a DCC set of nonnegative real numbers and $n$ a natural number. Suppose that $(X', B'+ M')$ is a generalized polarized pair with data $X \xrightarrow{f} X' \to Z$ and nef divisor $M$ on $X$. Assume that $D'$ on $X'$ is an effective $\Rr$-divisor and that $N$ on $X$ is an $\Rr$-divisor which is nef$/Z$ and that $D'+N'$ is $\Rr$-Cartier with $N' = f_*N$. Suppose that the following conditions are satisfied. 
\begin{enumerate}
\item $(X',B'+M')$ is generalized lc of dimension $n$,
\item $M= \sum \mu_j M_j$ where $M_j$ are nef$/Z$ Cartier divisors and $\mu_j \in \Lambda$,
\item $N= \sum \nu_k N_k$ where $N_k$ are nef$/Z$ Cartier divisors and $\nu_k \in \Lambda$, and
\item the coefficients of $B'$ and $D'$ belong to $\Lambda$.
\end{enumerate}
Then there is an ACC set  $\Theta$ depending only on $\Lambda, n$ such that the generalized lc threshold of $D'+ N'$ with respect to $(X', B'+ M')$ belongs to $\Theta$.
\end{theorem}

\begin{theorem}[Global ACC]\label{thm: global ACC}
Let  $\Lambda$ be a DCC set of nonnegative real numbers and $n$ a natural number. Then there is a finite subset $\Lambda^0 \subseteq \Lambda$ depending only on $\Lambda, d$ such that if $(X',B' + M')$ is a generalized pair with data $X \to X' \to Z$ and $M$ satisfying the following conditions,
\begin{enumerate}
\item $(X', B'+M')$ is generalized lc of dimension $n$,
\item $Z$ is a point,
\item $M= \sum \mu_j M_j$ where $M_j$ are nef Cartier divisors and $\mu_j \in \Lambda$,
\item $\mu_j =0$ if $M_j \equiv 0$,
\item the coefficients of $B'$ belong to $\Lambda$, and 
\item $K_{X'} +B' +M' \equiv 0$,
\end{enumerate}
then the coefficients of $B'$ and the $\mu_j$ belong to $\Lambda^0$.
\end{theorem}

The next result is \cite[Theorem 1.1]{Bir16b} which is known as Borisov-Alexeev-Borisov conjecture (BAB conjecture).

\begin{theorem}[BAB conjecture]\label{thm: BAB}
Let $n$ be a natural number and $\epsilon$ a positive real number. Then the projective varieties $X$ such that
\begin{enumerate}
\item $(X, B)$ is $\epsilon$-lc of dimension $n$ for some boundary $B$, and
\item $-(K_X + B)$ is nef and big,
\end{enumerate}
form a bounded family
\end{theorem}

Finally, for completeness, we mention the ACC for pseudo-effective thresholds (see \eqref{eq: pet}) of generalized polarized pairs. This will not be used in the rest of the paper and it can be obtained by the same argument as \cite[Theorem 1.4]{HL17}.

\begin{theorem}
Let $\Lambda$ be a DCC set of nonnegative real numbers and $n$ a natural number. Let $(X', B'+N')$ be a generalized polarized pair satisfying the following conditions.
\begin{enumerate}
\item $\dim X=n$, coefficients of $B'$ and $N'$ are in $\Lambda$,
\item $(X', B'+N')$ is generalized lc such that $B'$ is the boundary part and $N'$ the nef part, and
\item $M'$ is a nef and big Cartier divisor.
\end{enumerate}
Then the set of pseudo-effective thresholds of $M'$ with respect to $K_{X'}+ B'+N'$ is an ACC set.
\end{theorem}

\section{Accumulation points of pseudo-effective thresholds}\label{sec: accumulation points}

\subsection{A characterization of $k$-th accumulation points}\label{subsec: a characterization}

Let $M$ be a big $\Qq$-Cartier divisor. Define the pseudo-effective threshold of $M$ with respect to $D$ to be 
\begin{equation}\label{eq: pet}
\pet(D; M) \coloneqq \inf\{t \in \Rr_{\geq 0} \mid D + tM \text{~is effective}\}.
\end{equation} For a log pair $(X, \Delta)$, set the pseudo-effective threshold of $M$ with respect to $K_X+\Delta$ to be
\[
\pet(X, \Delta; M) \coloneqq \inf\{t \in \Rr_{\geq 0} \mid K_X+\Delta+tM \text{~is effective}\}.
\]

The following lemma gives extra flexibility on singularities when working with the accumulation points of pseudo-effective thresholds.

\begin{lemma}\label{le: perturbation of threshold}
Let $D, B$ be $\Qq$-Cartier divisors, and $M$ be a big $\Qq$-Cartier divisor. Then 
\[
\lim_{\epsilon \to 0} \pet(D+\epsilon B; M) = \pet(D; M).
\]
\end{lemma}
\begin{proof}
By considering $-B$, we can assume that $\epsilon \to 0^+$. Because $M$ is big, there is an $a \in \Rr_{> 0}$ such that $M-aB\equiv E\geq0$. By 
\[
D +\epsilon B \equiv D+\frac{\epsilon}{a}(M-E) \equiv D+\frac{\epsilon}{a}M - \frac{\epsilon}{a} E,
\] we have 
\[
\pet(D+\epsilon B;M)\geq \pet(D+\frac{\epsilon}{a}M;M) \geq \pet(D;M)-  \frac{\epsilon}{a}.
\] On the other hand, because there exists $b>0$ such that $b M + B \equiv F>0$,
\[
\pet(D +\epsilon B;M) \leq \pet(D- \epsilon b M;M) \leq \pet(D;M) + \epsilon b.
\] By taking $\epsilon \to 0^+$, we get the desired result.
\end{proof}

We use the following notation and conventions. Assume that $I\subseteq [0,1]$, and $c \in \Rr_{\geq 0}$. Recall that $I_+ = \{\sum n_i a_i \leq 1 \mid a_i \in I, n_i \in \Nn\} \cup \{0\}$. Set 
\[
\begin{split}
&D(I) = \{\frac{m-1+f}{m} \leq 1 \mid m \in \Nn, f \in I_+\}, and\\
&D_c(I) = \{\frac{m-1+f+kc}{m} \leq 1 \mid m, k \in \Nn, f \in I_+\}.
\end{split}
\]

By direct computations (for example, see \cite[Lemma 4.4]{MP04}), we have
\begin{equation}\label{eq: D(I)}
\begin{split}
& D(D(I)) = D(I)\cup \{1\}, \ D(D_c(I))\subseteq D_c(I)\cup \{1\},\\
& D_c(D(I)) \subseteq D_c(I)\cup \{1\}, \ D_c(D_c(I)) \subseteq D_c(I)\cup \{1\}.
\end{split}
\end{equation} Moreover, when $c>0$,
\[
D(D_c(I)) = D_c(I),\ D_c(D(I)) = D_c(I).
\]

For a divisor $\Delta$, we write $\Delta \in I$ if the coefficients of $\Delta$ lie in $I$. We are interested in the set of pseudo-effective thresholds
\[
\begin{split}
\PET_n(I) = \{ & \pet(X, \Delta; M) \mid (X, \Delta) \text{~is lc}, \Delta \in I, \\
&M \text{~is a nef and big Cartier divisor~}, \dim X=n\}.
\end{split}
\] Notice that the $\PET_n$ in Fujita's log spectrum conjecture \ref{conj: Fujita's log spectrum conjecture} is contained in $\PET_n(\{1\})$.

\medskip

Suppose that $X \to X'$ is a generalized polarized pair with $M'$ the nef part ($M'$ may not be effective) and $f_*M=M'$. For $c \in \Rr_{\geq 0}$, a generalized polarized pair $(X', B'+cM')$ is said to satisfy condition ($\dagger$) if the following conditions hold:

\begin{quote}
~($\dagger$)  
\begin{enumerate}
\item $(X', B' + cM')$  is generalized lc,
\item $K_{X'}+B' + cM' \equiv 0$, and  $B' \in D(I) \cup D_c(I)$, 
\item $M' = f_*M$ with $M$ a nef Cartier divisor, and $M'$ a $\Qq$-Cartier divisor,
\item if $M \equiv 0$, then at least one coefficient of $B'$ lies in $D_c(I)$.
\end{enumerate} 
\end{quote}

\medskip

The following sets will be considered in the sequel. First, for $c \in \Rr_{\geq 0}$, let
\begin{equation}\label{eq: set N}
\begin{split}
\mathfrak{N}_n(I,c) = &\{(X', B' + cM') \mid \dim X =n, \\
&\text{~and~} (X', B' + cM') \text{~satisfies condition ($\dagger$)} \}.
\end{split}
\end{equation}

Notice that when $M \equiv 0$ (or equivalently $M'\equiv 0$), this generalized polarized pair belongs to $\mathfrak{N}_n(I,c)$ defined in \cite{HMX14} Page 559. It is crucial to require $M$ to be Cartier. This property will be preserved under the generalized MMP and all the actions performed below. For $n \in \Nn$, set
\begin{equation}\label{eq: sets}
\begin{split}
\mathfrak{K}_n(I,c) = &\{(X', B' + cM') \mid (X', B' + cM') \in \mathfrak{N}_n(I,c)  \\
&\text{~is generalized klt, and~} \rho(X')=1\}\\
N_n(I) =  &\{c \in \Rr \mid \mathfrak{N}_m(I,c) \neq \emptyset, m \leq n \}\\
K_n(I) =  &\{c \in \Rr \mid \mathfrak{K}_m(I,c) \neq \emptyset, m \leq n \}.
\end{split} 
\end{equation} Notice that in $N_n(I)$ and $K_n(I)$, we also consider varieties of dimensions less than $n$. Besides, set
\begin{equation}\label{eq: N_0}
N_0(I)=K_0(I)=\{\frac{1-f}{m} \geq 0 \mid f\in I_+, m\in \Nn\}.
\end{equation} Notice that $N_0(I) \subseteq K_1(I) \subseteq K_n(I)$. In fact, by $\frac 1 2\in D(I)$, for any $f\in I_+, f \leq 1$, let $X'=\Pp^1, B'=\frac 1 2 (p_1+p_2)+f q$ and $M'=mz$, where $p_1, p_2, q, z \in X'$ are distinct closed points. Then $K_{X'}+B'+cM' \equiv 0$ for $c = \frac{1-f}{m}$.

\begin{lemma}\label{le: trivial nef part}
Let $I \subseteq [0,1], c \in \Rr_{\geq 0}$ and $(X', B'+cM')$ be a generalized lc pair with data $f:X\to X'$ and $cM$. Suppose that $B', M \in I$, $M'$ is $\Qq$-Cartier, and $K_{X'}+B'+cM' \equiv 0$. Let $S'$ be the normalization of an irreducible component of $\lfloor B' \rfloor$. Let $S$ be the strict transform of $S'$ on $X$, and
\[
(K_{X'}+B'+cM')|_{S'}=K_{S'}+B_{S'}+cM_{S'}
\] be the generalized adjunction. If $M'|_{S'} \not\equiv 0$, but $M|_S \equiv 0$, then $B_{S'} \in D(I) \cup D_{c}(I)$, and at least one coefficient is $\frac{m-1+f+kc}{m}$ with $m, k \in \Nn, f\in I_+$. Moreover, in this case, $(S', B_{S'})$ is lc and $K_{S'}+B_{S'} \equiv 0$.
\end{lemma}
\begin{proof}
By the generalized adjunction formula, we have $B_{S'} \in D(I) \cup D_{c}(I)$, and $(S', B_{S'})$ is lc because $({S'}, B_{S'}+cM_{S'})$ is generalized lc. Moreover, as $M|_S \equiv 0$, $M_{S'} \equiv 0$, we have $K_{S'}+B_{S'} \equiv 0$. We claim that there exists a coefficient $\frac{m-1+f+kc}{m}$ of $B_{S'}$ with $k>0$.

\medskip

By the negativity lemma, we have 
\begin{equation}\label{eq: negativity}
f^*M'=M+E
\end{equation} with $E$ an effective exceptional divisor. Let $g$ be the morphism $S \to S'$, then $(f^*M')|_S=g^*(M'|_{S'})$. Hence $g^*(M'|_{S'})=(f^*M')|_S=M|_S+E|_S$. As $M'|_{S'} \not\equiv 0$ and $M|_S \equiv 0$, 
\[
E|_S \equiv g^*(M'|_{S'})\not\equiv 0.
\] Thus $g_*(E|_S) \equiv M'|_{S'} \not\equiv 0$. Because $B_{S'} = g_*((B-S)|_S+cE|_S)$, we show the claim (see the last paragraph of Section \ref{subsec: adjunction}).
\end{proof}

The following lemma is in the same spirit of \cite[Lemma 11.4]{HMX14}.

\begin{lemma}\label{le: picard number 1}
Suppose that $1 \in I$, then $N_n(I) = K_n(I)$.
\end{lemma}
\begin{proof}
The ``$\supseteq$'' is by definition, we only need to show the inclusion ``$\subseteq$''. We do induction on $n$.

\medskip

When $n=1$, for $(X', B'+cM') \in \mathfrak N_1(I,c)$, $X'$ is an elliptic curve or $\Pp^1$. In the former case, $B'=0, M' \equiv 0$, $(X', B'+cM')$ is generalized klt, and $\rho(X')=1$. Thus $(X', B'+cM') \in \mathfrak K_1(I,c)$ and $c \in K_1(I)$. When $X'=\Pp^1$, for $(X', B'+cM') \in \mathfrak N_1(I,c)$, if $(X', B'+cM')$ is generalized lc but not generalized klt, then at least one coefficient of $B'$ is $1$. By $0 \in I_+$, $\frac 1 2 \in D(I)$. If $p_1,p_2,p_3,p_4, q \in X'$ are distinct closed points, then for any $c \in \Rr_{\geq 0}$, $(X', B'+cM') \in \mathfrak K_1(I,c)$ for $B' = \frac 1 2(p_1+p_2+p_3+p_4)$ and $M'=cq$. In particular, $c=0 \in K_1(I)$. We can assume that $c>0$.

\medskip

First, we consider the case $M \not\equiv 0$. Write $B'=p+B''$, where $p$ has coefficient $1$. Choose $p_1, p_2 \in X'$ to be two closed points different from $\Supp B''$, then
\begin{equation}\label{eq: B'}
B'=p+B''\equiv \frac 1 2 p_1+\frac 1 2 p_2+B''.
\end{equation} We can continue this process if there exists a component of $B''$ whose coefficient is $1$. By doing this, we get $\tilde B' \in D(I) \cup D_c(I)$ such that $B'\equiv \tilde B'$ and $(X', \tilde B'+cM')$ is generalized klt. Thus $c \in K_1(I)$. 

\medskip

Next, we consider the case $M \equiv 0$. Then $B'$ has a coefficient $\frac{m-1+f+kc}{m}$ with $k\in\Nn$. If $\frac{m-1+f+kc}{m}<1$, then we can do the same thing as \eqref{eq: B'} and obtain a $\tilde B'$ so that $(X', \tilde B'+cM')$ is generalized klt with $\frac{m-1+f+kc}{m}$ to be a coefficient of $\tilde B'$. Thus $c \in K_1(I)$. If $\frac{m-1+f+kc}{m}=1$, then $c=\frac{1-f}{k}>0$. Consider $B'=\frac 1 2 p_1+\frac 1 2 p_2+f p_3$ and $M'=kq$, then $(X', B'+cM') \in \mathfrak K_1(I,c)$ for $c=\frac{1-f}{k}$. Thus  $c=\frac{1-f}{k} \in K_1(I)$. This completes the $n=1$ case.

\medskip

From now on, assume that the inclusion holds for any dimension less than $n$.

\medskip

Let $(X', B'+cM') \in \mathfrak{N}_n(I,c)$. If $M' \equiv 0$, then this is just \cite[Lemma 11.4]{HMX14}, hence we can assume that $M' \not\equiv 0$.

\medskip

By Proposition \ref{prop: dlt}, there exists a generalized lc pair $(X'', B''+cM'')$ such that 
\[
K_{X''}+B''+cM'' =f^*(K_{X'}+ B'+cM') \equiv 0,
\] with $(X'',0)$ has klt singularities, hence it satisfies property (ii) in \S 2.3. Let $T''=\lfloor B'' \rfloor$. If $T''=0$, then $(X'', B''+cM'')$ is generalized klt by Proposition \ref{prop: dlt} (5). Moreover, $M''$ can be obtained as a push-forward of a Cartier divisor from a common resolution. Thus $(X'', B''+cM'') \in \mathfrak{N}_n(I,c)$. 

\medskip

By Lemma \ref{lem: MMP} (or just \cite{BCHM10}), we can run a $(K_{X''}+B'')$-MMP with scaling, which is the same as a $(-cM'')$-MMP. As $-cM''$ is not pseudo-effective, we can obtain a Mori fibre space $Y \to Z$, and $-M''$ is not contracted in this MMP. Let $(Y, B_Y + cM_Y)$ be the log pair obtained above. There are three possibilities. 

\medskip

\noindent Case (1). When $T''=0$, then $(Y, B_Y + cM_Y)$ is a generalized pair with generalized klt singularities. Taking a general fibre $F$ and restricting to $F$, we have 
\[
K_F + B_Y|_F + cM_Y|_F \equiv 0,
\] with $M_Y|_F\not\equiv 0$ (because $M_Y$ is positively intersects with curves contracted to $Z$). $(F, B_Y|_F+cM_Y|_F)$ is still generalized klt with $M_Y|_F$ a push-forward of a Cartier divisor. When $\dim Z>0$, then $\dim F<\dim X$, and we get the result by induction hypothesis. Otherwise, $F=Y$ has Picard number $1$, and thus $(F, B_Y|_F + cM_Y|_F) \in \mathfrak{K}_n(I)$. Hence $c\in K_n(I)$.

\medskip

\noindent Case (2). When $T'' \neq 0$, and $T''$ is not contracted in $X'' \to Y$. We do the same thing as Case (1) when $\dim Z>0$, and obtain $(F, B_Y|_F + cM_Y|_F) \in \mathfrak{N}_{\dim F}(I)$, hence the result holds by induction hypothesis. When $\dim Z =0$, then $\rho(Y)=1$, and as $M_Y \not\equiv 0$, $M_Y|_{S_Y} \not\equiv 0$, where $S_Y$ is an irreducible component of the push-forward of $T''$. Hence we do the generalized adjunction on the normalization $S$ of $S_Y$,
\[
(K_{Y}+B_Y+cM_Y)|_{S} = K_S+B_S+cM_S \equiv 0.
\] By generalized adjunction and \eqref{eq: D(I)}, $B_S \in D(D(I)\cup D_c(I)) \subseteq D(I)\cup D_c(I)$. Either $M_{S} \not\equiv 0$, or $M_{S} \equiv 0$ but $M_Y|_S \neq 0$. In the later case, by Lemma \ref{le: trivial nef part}, we have at least one component of $B_S$ has coefficient in $D_c(D(I)\cup D_c(I))\subseteq D_c(I)$. Hence, for both cases, we are done by induction hypothesis.

\medskip

\noindent Case (3). When $T''$ is contracted in some step of the above MMP. Without loss of generality, we can assume that a component $S''$ is contracted in $X'' \to Y'$. Thus, by the negativity lemma, $S''$ is a covering family of curves $\{C\}$ such that $S'' \cdot C<0$. But $M''\cdot C>0$ as we run $(-cM'')$-MMP, and thus $M''|_{S''} \not\equiv 0$. Hence as Case (2), we do adjunction on the normalization of $S''$ and complete the argument by induction hypothesis.
\end{proof}

\begin{proposition}\label{prop: limit in numerically trivial}
Assume $1 \in I$, then $\lim^1 \PET_n(I) \subseteq \lim^1 N_n(I)$. 
\end{proposition}

\begin{proof} 
We proof the claim by induction on $n$. For $n=1$, $\PET_1(I) \subseteq N_1(I)$ by definition, hence the claim holds. We can assume that the claim holds for any dimension less than $n$.

\medskip

\noindent Step 1. Let $c_i= \pet (X_i, \Delta_i; M_i)$. Suppose that $\{c_i\}_{i\in\Nn}$ has an accumulation point $c$. We can assume that $c_i>0$. Taking a dlt modification $f_i$ of $(X_i, \Delta_i)$
\[
K_{W_i}+ \Delta_{W_i}+T_i = f^*_i(K_{X_i}+\Delta_i),
\] where $T_i$ is a reduced divisor, $W_i$ is $\Qq$-factorial and $({W_i}, \Delta_{W_i}+T_i)$ is dlt. Moreover, $c_i$ is the same as the pseudo-effective threshold of $f_i^*M_i$ with respect to $K_{W_i}+ \Delta_{W_i}+T_i$. Replacing by $({W_i}, \Delta_{W_i}+T_i)$, we can assume that $X_i$ is $\Qq$-factorial and $(X_i, \Delta_i)$ is dlt. 

\medskip

\noindent Step 2. Choose an ample divisor $A_i$ and $1 \gg \epsilon_i>0$ such that $(X_i, \Delta_i+ \epsilon_i A_i)$ is lc. Then there exists a $\tilde \Delta_i$ satisfying
\[
K_{X_i}+ \Delta_i +\epsilon_i A_i \sim_\Rr K_{X_i}+\tilde \Delta_i
\] such that $(X_i, \tilde \Delta_i)$ is klt. By Lemma \ref{le: perturbation of threshold}, for any $1 \gg \delta_i>0$, we can find a sufficiently small $\epsilon_i>0$, such that
\[
c_{i} > c_i' \coloneqq \pet(X_i, \Delta_i + \epsilon_i A_i; M_i) \geq c_i-\delta_i.
\] 

Because $M_i$ is nef, $(X_i, \tilde \Delta_i+tM_i)$ is generalized klt for any $t \geq 0$ with data $X_i \xrightarrow{\rm id} X_i$ and $M_i$. In the same way, $(X_i, \Delta_i+\epsilon A_i+tM_i)$ is generalized lc. By Lemma \ref{lem: MMP} (2), we can run a $(K_{X_i} + \tilde \Delta_i + c'_i M_i)$-MMP, $X_i \dashrightarrow Y_i$.  As $K_{X_i}  + \tilde \Delta_i + c'_i M_i$ is pseudo-effective, $Y_i$ is birational to $X_i$. Moreover, $K_{Y_i} +\tilde \Delta_{Y_i} + c'_i M_{Y_i}$ is nef but not big as $c'_i$ is the pseudo-effective threshold. Then by Lemma \ref{lem: MMP} (2) again, $K_{Y_i} + \tilde \Delta_{Y_i} + c'_i M_{Y_i}$ is semi-ample, and thus it defines a fibration $g_i: Y_i \to Z_i$. Taking a general fibre $F_i$ (if $Z_i$ is a point, then $F_i=Y_i$), and restricting to $F_i$, we get
\begin{equation}\label{eq: on fibre}
(K_{Y_i} + \tilde \Delta_{Y_i} + c'_i M_{Y_i})|_{F_i} = K_{F_i}+ \tilde\Delta_{Y_i}|_{F_i} + c'_i M_{Y_i}|_{F_i} \equiv 0.
\end{equation} Notice that because $M_i$ is big (by the definition of $\PET_n(I)$), $M_{Y_i}|_{F_i}$ is big. Moreover, as $(X_i, \Delta_i + \epsilon A_i+c'_iM_i)$ is generalized lc, $(Y_i, \Delta_{Y_i}+\epsilon A_{Y_i}+c'_iM_{Y_i})$ is also generalized lc, where $\Delta_{Y_i}, A_{Y_i}$ are push-forwards of $\Delta_i, A_i$. In fact, the MMP above is also an MMP on 
$K_{X_i}+\Delta_{i}+\epsilon A_{i}+c'_iM_{i}$. The restriction to a general fibre preserves the generalized lc property. Moreover, as $K_{Y_i}+\Delta_{Y_i}+c_iM_{Y_i}$ is the push-forward of $K_{X_i}+\Delta_i+c_iM_i$, it is pseudo-effective and so is $K_{F_i}+\Delta_{Y_i}|_{F_i}+c_i M_{Y_i}|_{F_i}$. By \eqref{eq: on fibre},
\begin{equation}\label{eq: inequality}
c_i \geq \tau_i \coloneqq \pet(F_i, \Delta_{Y_i}|_{F_i}; M_{Y_i}|_{F_i}) > c_i' \geq c_i - \delta_i.
\end{equation} In particular, $\lim_{i \to \infty} \tau_i = \lim_{i \to \infty} c_i$. Moreover, $(F_i, \Delta_{Y_i}|_{F_i}+c_i' M_{Y_i}|_{F_i})$ is generalized lc.

\medskip

\noindent Step 3. First suppose that $\dim F_i <\dim Y_i$. If $(F_i, \Delta_{Y_i}|_{F_i}+\tau_i M_{Y_i}|_{F_i})$ is generalized lc, then the result follows from the induction hypothesis. If $(F_i, \Delta_{Y_i}|_{F_i}+\tau_i M_{Y_i}|_{F_i})$ is not generalized lc, let $\tau_i'$ be the generalized log canonical threshold of $M_{Y_i}|_{F_i}$ with respect to $(F_i, \Delta_{Y_i}|_{F_i})$, then $\tau_i > \tau_i' \geq c_i'$ ($\tau_i' \geq c_i'$ because $(F_i, \Delta_{Y_i}|_{F_i}+c_i'M_{Y_i}|_{F_i})$ is generalized lc). By Proposition \ref{lem: dlt for klt}, and notice that $(F_i, 0)$ is klt, we can extract a generalized lc place $S_i$ such that
\[
K_{F''_i} +S_i+ \Delta_{F''_i}+\tau_i'M_{F''_i} = p^*(K_{F_i}+\Delta_{Y_i}|_{F_i}+\tau_i'M_{Y_i}|_{F'_i}),
\] where $\Delta_{F''_i}$ is the strict transform of $\Delta_{Y_i}|_{F_i}$. Moreover, the relative Picard number $\rho(F_i''/F_i')=1$ and $S_i$ is anti-ample/$F_i'$. Suppose that $V_i$ is the generalized lc center of $S_i$ with $\tilde V_i$ its normalization. Let $S_i \to \tilde V_i$ be the corresponding morphism, and $\Upxi_i$ be a general fibre. Then we have
\[
K_{\Upxi_i}+\Delta_{\Upxi_i}+\tau_i'M_{\Upxi_i} = ((K_{F''_i} +S_i+ \Delta_{F''_i}+\tau_i'M_{F''_i})_{S_i})|_{\Upxi_i} \equiv 0.
\] For $(\Upxi_i, \Delta_{\Upxi_i}+\tau_i'M_{\Upxi_i})$, it could happen that $M_{\Upxi_i}\equiv 0$. But if this is the case, then at least one component of $\Delta_{\Upxi_i}$ lies in $D_{\tau'_i}(I)$. The verification of the claim is identical to that in Step 5 of Proposition \ref{prop: induction of accumulation point} (from \eqref{eq: numerically trivial} to the end, especially Case (b1)), and thus we leave the details to that argument. Now $({\Upxi_i}, \Delta_{\Upxi_i}+\tau_i'M_{\Upxi_i})$ is a generalized lc pair satisfying condition ($\dagger$) such that $\lim_{i \to \infty} \tau_i' = c$. Then by the induction hypothesis, we are done.

\medskip

\noindent Step 4. Now, suppose that $F_i = Y_i$, then $\rho(Y_i)=1$, and $K_{Y_i}+\Delta_{Y_i}+\tau_iM_{Y_i} \equiv 0$. If $({Y_i}, \Delta_{Y_i}+\tau_iM_{Y_i})$ is generalized lc, then $\tau_i \in N_n(I)$. Otherwise, let $\tau_i'$ be the generalized lc threshold of $M_{Y_i}$ with respect to $(Y_i, \Delta_{Y_i})$. Again, $\tau_i>\tau_i' \geq c_i'$. Then the same argument as Step 3 gives the desired result.
\end{proof}

\begin{proposition}\label{prop: induction of accumulation point}
Let $I$ be a DCC set such that $I=I_+$. Assume that $1\in I$ with $1$ the only possible accumulation point. Let $n \in \Nn$ be a fixed integer, and $\{c_i\}_{i\in \Nn}$ be a strictly decreasing sequence with limit $c, c> 0$. Suppose that $(X'_i, \Delta'_i)$ satisfies the following conditions.
\begin{enumerate}
\item $\dim X_i'=n$,
\item $K_{X'_i} + \Delta'_i \equiv 0$ with $\Delta'_i = A'_i + B'_i + c_iM'_i$, 
\item $(X'_i, A'_i + B'_i + c_i M'_i)$ is generalized lc with $A'_i + B'_i $ the boundary part and $c_i M'_i$ the nef part, $M_i'$ is a $\Qq$-Cartier divisor,
\item the coefficients of $A'_i$ are either $0$ or approaching $1$,
\item $B'_i \in D(I) \cup D_{c_i}(I)$, and
\item $M'_i$ is a push-forward of a nef Cartier divisor $M_i$, such that either $M_i' \not\equiv 0$, or when $M_i' \equiv 0$, at least one component of $B'_i$ lies in $D_{c_i}(I)$.
\end{enumerate}
Then $c \in N_{n-1}(I)$.
\end{proposition}

\begin{proof}
We do induction on $n$. 

\medskip

First, for the base case $n=1$, by $c>0$, we have $X_i'=\Pp^1$. Let $a_i = \deg A_i', l_i =\deg M_i'$ and $\frac{m_i^{(j)}-1+f_i^{(j)}+k_i^{(j)}c_i}{m_i^{(j)}}$ be a coefficient of $B_i'$, then
\begin{equation}\label{eq: sum=0}
-2+a_i+\sum_j \frac{m_i^{(j)}-1+f_i^{(j)}+k_i^{(j)}c_i}{m_i^{(j)}}+l_i c_i=0.
\end{equation} Notice that $\lim_{i\to +\infty} a_i = 0,1$ by assumption (4) and $c>0$, $l_i \in \Zz_{\geq 0}$ and if $l_i=0$, then there exists $k_i^{(j)}>0$ by assumption (6). By $c>0$, $k_i^{(j)}, l_i$ are bounded above, and thus by passing to a subsequence, we can assume that they are equal to $k$ and $l$ respectively. Because $c$ is an accumulation point, by passing to a subsequence, we can assume that $c_i$ are distinct. Hence if $\lim_{i\to +\infty} a_i = 0$ (i.e. $\deg A_i=0$), then $m_i^{(j)}$ or $f_i^{(j)}$ does not belong to a finite set. Otherwise, $c_i$ also belongs to a finite set. Because $1$ is the only possible accumulation point of $I=I_+$, if $f_i^{(j)}$ does not belong to a finite set, then by passing to a subsequence, $\lim_{i \to +\infty}f_i^{(j)}=1$. Thus one coefficient $\frac{m_i^{(j)}-1+f_i^{(j)}+k_i^{(j)}c_i}{m_i^{(j)}}$ of $B_i'$ has limit $1$. Taking $i \to +\infty$, we have
\[
-1+lc=0~\text{or}~-1+\frac{m-1+f+kc}{m}+lc=0,
\] where $f \in I_+=I$, $m \in \Nn$, $k,l\in \Zz_{\geq 0}$ and at least one of $k, l$ is non-zero. Thus $c=\frac 1 l$ or $\frac{1-f}{k+ml}$, which belongs to $N_0(I)$ (see \eqref{eq: N_0}). This shows $n=1$ case. From now on, we assume that the result holds for any dimension less than $n$.

\medskip

Next, the problem can be simplified by the following. The $M_i' \equiv 0$ case is just \cite[Proposition 11.7]{HMX14}. In fact, $M_i =f_i^*M_i'-E \equiv -E$ with $E \geq 0$ (see \eqref{eq: negativity}). If $E>0$, then $M_i$ cannot be nef. Thus $M_i \equiv 0$. Besides, according to (6), at least a component of $B_i'$ lies in $D_{c_i}(I)$, by $\frac{m_i-1+f_i+k_ic_i}{m_i} \leq 1$, we have $c_i\leq 1$. Notice that the $B_i+C_i$ in \cite[Proposition 11.7]{HMX14}  is $B_i$ in our case. Hence, we assume $M_i'\not\equiv 0$ in the following.

\medskip

We can assume that $A'_i, B'_i$ do not have common components. Indeed, if there is a component of $B'_i$ which approaching $1$, then we add it to $A'_i$. Besides, we can assume that in the coefficients $\frac{m_i-1+f_i+k_ic_i}{m_i}$ of $B'_i$, $m_i$ are bounded. Because $1$ is the only possible accumulation point of $I=I_+$, if $f_i<1$ is strictly increasing with limit $f$, then $f=1$, and thus $\frac{m_i-1+f_i+k_ic_i}{m_i}$ approaching $1$. Hence, by passing to a subsequence, we can assume that those $f_i$ appearing in the coefficients of $B_i'$ are chosen from a finite set. 

\medskip

\noindent Step 1. In this step, we reduce the problem to $\Qq$-factorial Fano and Picard number $1$ case. 

\medskip

By Proposition \ref{prop: dlt}, there exists $f: X_i'' \to X_i'$ such that
\[
K_{X_i''}+A_i''+ B_i'' +  c_i M_i'' = f^*(K_{X'_i}+A_i'+B_i' +c_iM_i'),
\] and $(X_i'', A_i'' + B_i''+c_iM_i'')$ is still generalized lc, $X_i''$ is $\Qq$-factorial, and $(X_i'',0)$ is klt with $A_i'', B_i'' \in I$. We can run a $(K_{X_i''}+A_i''+ B_i'')$-MMP by Lemma \ref{lem: MMP} (1), $X_i \dashrightarrow Y_i$, which is the same as a $(-M_i'')$-MMP and thus we can assume that there exists a Mori fibre space $Y_i \to Z_i$. When $\dim Z_i>0$ for infinitely many $i$, we take a general fibre and restricting everything to this fibre. Then we complete the proof by induction hypothesis. Otherwise, we can assume $\dim Z_i=0$ for each $i$, and thus $\rho(Y_i)=1$. Replacing $(X'_i, \Delta'_i)$ by $(Y_i, {f_i}_*\Delta'_i)$, we can assume that $\rho(X'_i)=1$. In particular, we can assume that $X'_i$ is a $\Qq$-factorial Fano variety and $M'_i$ is big.

\medskip

\noindent Step 2. In this step, we show that there is no $\epsilon>0$ such that $(X'_i, \Delta'_i)$ is generalized $\epsilon$-lc for each $i$. 

\medskip

Suppose that there exists $\epsilon>0$ such that for infinitely many $i$, the total generalized log discrepancy of $(X'_i, \Delta'_i)$ is greater than $\epsilon$. As the total generalized log discrepancy of $(X'_i, A'_i+B'_i)$ is no less than that of $(X'_i, \Delta'_i)$, $(X'_i, A'_i+B'_i)$ is $\epsilon$-lc. By Theorem \ref{thm: BAB}, such $X'_i$ forms a bounded family. Moreover, as coefficient of $A'_i$ are either $0$ or approaching $1$, by passing to a tail, we can assume $A'_i=0$ (otherwise the total generalized log discrepancies will $<\epsilon$). The coefficients of $B'_i$ are of the form $\frac{m_i-1+f_i+k_ic_i}{m_i}$ with $f_i\in I_+$, by passing to a subsequence, we can assume that $m_i$ is fixed. As above, $f_i$ is chosen from a finite set. By considering a very ample divisor on the bounded family and by passing to a subsequence again, we see that $(K_{X'_i} + \Delta'_i)\cdot H^{n-1}=0$. Moreover,  $K_{X'_i}\cdot H^{n-1}, B'_i \cdot H^{n-1}$ are bounded. $M'_i\cdot H^{n-1}$ is also bounded because $\lim c_i =c \neq 0$ and $M'_i$ is a prime divisor as it is a push-forward of a Cartier divisor. Notice that if $D \subseteq X_i'$ is a Weil divisor and $H$ is a very ample divisor, then $D \cdot H^{n-1} \in \Zz$. For $\frac{m_i-1+f_i+k_ic_i}{m_i}$, those $m_i,k_i,f_i$ are bounded and thus are fixed by passing to a subsequence. This contradicts to  
the strictly decreasing of $c_i$ with limit $c\neq 0$. Thus we can assume that for any $\epsilon>0$, there exists $(X'_i, \Delta'_i)$ whose total generalized log discrepancy is less than $\epsilon$.

\medskip

\noindent Step 3. In this step, we show that $A_i'$ can be assumed to be non-zero. 

\medskip

If $A'_i=0$, then by Proposition \ref{prop: dlt}, the definition of generalized log discrepancy and the convention before Step 1, there exists $g_i: Y_i \to X_i'$ such that
\begin{equation}\label{eq: extract T}
K_{Y_i}+B_{Y_i}+E_i +a_iT_i+ c_iM_{Y_i}= g_i^*(K_{X'_i}+B'_i+c_iM'_i),
\end{equation}  where $E_i$ is a reduced divisor (possibly be $0$), and $1 \geq a_i\geq 1-\epsilon$. By putting $A_i=E_i +a_iT_i$, we can assume that $A_i \neq 0$. Moreover, when $\lfloor A_i \rfloor = 0$, $({Y_i}, B_{Y_i}+E_i +a_iT_i+ c_iM_{Y_i})$ is generalized klt. Replacing $(X_i', \Delta_i')$ by $(Y_i, B_{Y_i}+E_i+a_iT_i+c_iM_{Y_i})$, we can assume that $A_i \neq 0$, and when $\lfloor A_i \rfloor = 0$, $(X'_i, \Delta'_i)$ is generalized klt. Notice that in this case, $\rho(X_i')$ may not be $1$, but $M_i'$ is still big.

\medskip

Run a $(K_{X'_i}+B'_i+c_iM'_i)$-MMP, $f_i: X'_i \dashrightarrow Y_i$, and as $K_{X'_i}+B'_i+c_iM'_i\equiv -A'_i$ is not pseudo-effective, we can assume that the MMP terminates with a Mori fibre space $Y_i \to Z_i$. Let $F_i$ be a general fibre, because $M'_i$ is big, ${f_i}_*M'_i, {f_i}_*M'_i|_{F_i} \neq 0$. If $\dim Z_i>0$, then we are done. Thus we can assume that $\dim Z_i =0$, and thus $\rho(Y_i)=1$. Replacing by $Y_i, {f_i}_*A'_i, {f_i}_*B'_i, {f_i}_*M'_i$ , we can assume that $\rho(X_i')=1$. Moreover, $A'_i$ is not contracted by $f_i$ because we run a $(-A'_i)$-MMP.

\medskip

\noindent Step 4. In this step, we show the claim when $\lfloor A'_i \rfloor \neq 0$. 

\medskip

Suppose $\lfloor A'_i \rfloor \neq 0$, let $S'_i$ be the normalization of an irreducible component of $\lfloor A'_i \rfloor$. Because $\rho(X'_i)=1$, $M'_i|_{S_i'} \not \equiv 0$. By generalized adjunction (see Definition \ref{def: adjunction})
\[
(K_{X'_i}+A'_i+B'_i+c_iM'_i)|_{S'_i} = K_{S'_i}+A_{S_i'}+B_{S'_i}+c_iM_{S'_i} \equiv 0.
\] To be precise, suppose that $h_i: X_i \to X_i'$ is a log resolution, then 
\[
(K_{X_i}+A_i+B_i+c_iM_i) = h_i^*(K_{X'_i}+A'_i+B'_i+c_iM'_i)
\] with $A_i$ the strict transform of $A_i'$. Let $g_i$ denote the restriction of $h_i$ to $S_i \to S_i'$, then
\[
A_{S'_i} = {g_i}_*((A_i-S_i)|_{S_i}), \quad B_{S'_i}={g_i}_*(B_i|_{S_i}), \quad M_{S'_i}={g_i}_*(M_i|_{S_i}).
\]
Hence, $M_{S'_i}$ is the push-forward of the nef and Cartier divisor $M_{i}|_{S_i}$. $A_{S'_i}$ is either $0$ or approaching $1$. $B_{S_i'} \in D(I) \cup D_{c_i}(I)$, and by Lemma \ref{le: trivial nef part}, when $M_i|_{S_i} \equiv 0$, at least one coefficients of $B_{S'_i}$ is in $D_{c_i}(I)$. The $(S_i', A_{S'_i}+B_{S'_i}+c_iM_{S'_i})$ is still generalized lc, but $M_{S'_i}$ may not necessarily be $\Qq$-Cartier. By Proposition \ref{prop: dlt}, there exists a $\Qq$-factorial variety $\tilde S_i$ with a birational morphism $\pi_i: \tilde S_i \to S_i'$ such that 
\[
K_{\tilde S_i} + A_{\tilde S_i}+B_{\tilde S_i}+c_iM_{\tilde S_i} = \pi_i^*(K_{S_i'}+A_{S'_i}+B_{S'_i}+c_iM_{S'_i}).
\]
The generalized polarized pair $(\tilde S_i, A_{\tilde S_i}+B_{\tilde S_i}+c_iM_{\tilde S_i})$ satisfies all the assumptions of the proposition with $\dim \tilde S_i=n-1$. Hence by the induction hypothesis, we have $c \in N_{n-2}(I) \subseteq N_{n-1}(I)$.

\medskip

\noindent Step 5. In this step, we deal with $\lfloor A'_i \rfloor = 0$ case. We create a reduced divisor on $X_i'$ and discuss some of its properties.

\medskip

If $\lfloor A'_i \rfloor = 0$, then $(X'_i, \Delta'_i)$ is generalized klt by Step 3. As $c < c_i$, $({X'_i}, A'_i+B'_i+cM'_i)$ is also generalized klt.  Let $T'_i$ be the reduced divisor which is the support of $A'_i$. We claim that by passing to a subsequence, $({X'_i}, T'_i+B'_i+cM'_i)$ is still generalized lc. Otherwise, let $t_i$ be the generalized lc threshold of $T'_i$ with respect to $(X'_i, B'_i+cM'_i)$. Then $t_i<1$. But $t_i$ is greater or equal to the minimal coefficients of $A_i$. As the coefficients of $A_i$ is approaching $1$, $t_i$ is approaching $1$. This contradicts to the ACC for generalized lc thresholds (Theorem \ref{thm: ACC for glct}). By passing to a subsequence, we can assume that $({X'_i}, T'_i+B'_i+cM'_i)$ is generalized lc for each $i$.

\medskip

Now we consider the generalized polarized pair $({X'_i}, T'_i+B'_i+c_iM'_i)$. Until the end of this step, we show that when $({X'_i}, T'_i+B'_i+c_iM'_i)$ is not generalized lc for infinite $i$, the proposition holds. The remaining case when $({X'_i}, T'_i+B'_i+c_iM'_i)$ is generalized lc will be discussed in Step 5.

\medskip

Because $({X'_i}, T'_i+B'_i+c_iM'_i)$ is not generalized lc, $M_i' \not\equiv 0$, and let 
\[
c_i'=\sup\{t \in \Rr \mid K_{X'_i}+T'_i+B'_i+tM'_i\text{~is generalized lc~}\}.
\] Thus $c\leq c_i'<c_i$. Moreover, there exists a generalized lc center $V_i$ which is strictly contained in $M'_i$. By this we mean the following. Suppose that $h_i: X_i \to X_i'$ is a log resolution, then
\[
K_{X_i}+T_i+B_i(t)+tM_i = h_i^*(K_{X'_i}+T'_i+B'_i+tM'_i)
\] with $T_i$ the strict transform of $T_i'$. When $t=c_i'$, there exists at least one exceptional divisor whose coefficient in $B_i(c_i')$ is one and it is a generalized lc place over $V_i$. In particular, they cannot be components of $M_i'$. Besides, in $h_i^*M_i' = M_i+E_i$, we must have $E_i>0$. In fact, $E_i \geq 0$ by the negativity lemma, and if $E_i=0$,  then $({X'_i}, T'_i+B'_i+c_iM'_i)$ is already generalized lc. By the negativity lemma again, there exists a component $S_i$ of $E_i$ which is a covering family of curves $\{C\}$, such that ${h_i}_*C=0$ and $E_i \cdot C<0$. Thus $M_i \cdot C>0$ when $\Supp E_i = S_i$. 

\medskip

By Proposition \ref{lem: dlt for klt}, and notice that $(X',0)$ is klt, we can just extract $S_i$. That is, 
\[
K_{X_i''}+T_i''+B''_i + c_i'M_i''= \phi_i^*(K_{X_i'}+T_i'+B'_i + c_i'M_i'),
\] such that the push-forward of $S_i$ on $X_i''$, $S''_i$, is a component of $\lfloor B_i'' \rfloor$. Moreover, the relative Picard number $\rho(X_i''/X_i')=1$ and $S''_i$ is anti-ample/$X_i'$. On the other hand, we can assume that $X_i \xrightarrow{g_i} X_i'' \xrightarrow{\phi_i} X_i'$ and $h_i = \phi_i \circ g_i$. Thus $M_i'' = {g_i}_*M_i$ satisfies 
\[
\phi_i^*M_i'= {g_i}_*(h_i^*M_i') = M_i''+{g_i}_*E_i = M_i''+a_iS''_i,
\] for some $a_i>0$. We see that $M''_i$ is ample/$X_i'$. In particular, for a general fibre $F''_i$ of $S''_i \to V_i$, $M''_i|_{F''_i}$ is ample. 

\medskip

We have
\begin{equation}\label{eq: numerically trivial}
\begin{split}
& \left((K_{X_i''}+T_i''+B''_i + c_i'M_i'')|_{S''_i}\right)|_{F''_i}\\
 = & \left(\phi_i^*(K_{X_i'}+T_i'+B'_i + c_i'M_i')\right)|_{F''_i}\equiv 0.
\end{split}
\end{equation} By generalized adjunction, we have
\[
(K_{X_i''}+T_i''+B''_i + c_i'M_i'')|_{S''_i} = K_{S''_i}+T_{S''_i}+B_{S''_i}+c_i'M_{S''_i},
\] where $M_{S_i''}$ is the push-forward of $M_i|_{S_i}$. To remedy the problem that $M_{S_i''}$ may not be $\Qq$-Cartier, as Step 4, we can pass to a $\Qq$-factorial model $\tilde S_i$ over $S_i''$ and take the corresponding general fibre. Hence, without loss of generality, we can assume that $M_{S_i''}$ is $\Qq$-Cartier. After doing this, $M''_i|_{F''_i}$ is nef and big. Let $K_{F''_i}, T_{F''_i}, B_{F''_i}, M_{F''_i}$ be the restrictions of $K_{S''_i}, T_{S''_i}, B_{S''_i}, M_{S''_i}$ to $F_i''$ respectively. Then by \eqref{eq: numerically trivial}, we have
\[
K_{F''_i} +T_{F''_i} + B_{F''_i}+c_i' M_{F''_i} \equiv 0.
\] 

We need a detailed analysis on $B_{F''_i}$ and $M_{F''_i}$. Recall that $g_i: X \to X_i''$, and we let $q_i: S_i \to S_i'', \psi_i: F_i \to F_i''$ be the corresponding restrictions of $g_i$, where $F_i$ is the preimage of the general fibre $F_i''$. 

\[
\begin{tikzcd}
F_i \arrow[r] \arrow[d, "\psi_i"]
&S_i \arrow[r] \arrow[d, "q_i"]
& X \arrow[d, "g_i"] \\
F_i'' \arrow[r]
&S_i'' \arrow[r]
& X_i'' \arrow[d, "\phi_i"]\\
& & X_i'
\end{tikzcd}
\]

First, we claim that 
\begin{equation}\label{eq: pushforward}
M_{F''_i}={\psi_i}_*(M_i|_{F_i}).
\end{equation} In fact, by definition ${q_i}_*(M_i|_{S_i}) = M_{S_i''}$, thus $q_i^*(M_{S_i''}) = M_i|_{S_i} + \Theta_i$, where $\Theta_i$ is $q_i$-exceptional. We have
\[
\psi_i^*(M_{F_i''})=q_i^*(M_{S_i''})|_{F_i} = (M_i|_{S_i})|_{F_i} + \Theta_i|_{F_i} = M_i|_{F_i} +  \Theta_i|_{F_i},
\] where $\Theta_i|_{F_i}$ is $\psi_i$-exceptional, and hence the claim.

\medskip

Let $g_i^*M_i''=M_i+R_i$ with $R_i \geq 0$ a $g_i$-exceptional divisor. Then
\begin{equation}\label{eq: R}
g_i^*(M_i''|_{S_i''})=(g_i^*M_i'')|_{S_i}=M_i|_{S_i}+R_i|_{S_i}.
\end{equation} 

(a). When $R_i|_{S_i}$ is an exceptional divisor for $q_i: S_i \to S_i''$, then $R_i|_{F_i}=(R_i|_{S_i})|_{F_i}$ is $\psi_i$-exceptional as $F_i$ is a general fibre. Restricting \eqref{eq: R} to $F_i$, we have
\begin{equation}\label{eq: fibre1}
\psi_i^*(M_i''|_{F_i''})=(g_i^*M_i'')|_{F_i}=M_i|_{F_i}+R_i|_{F_i}.
\end{equation} As $M_i''|_{F_i''}$ is big, by \eqref{eq: pushforward}, $M_{F''_i}={\psi_i}_*(M_i|_{F_i})$ is big.

\medskip

(b). When $R_i|_{S_i}$ is not an exceptional divisor for $q_i: S_i \to S_i''$, write 
\[
R_i|_{S_i} = E_{S_i}+T_{S_i},
\] where $E_{S_i}$ is the summation of exceptional divisors in $R_i|_{S_i}$ and $T_{S_i}$ is the summation of the non-exceptional divisors in $R_i|_{S_i}$. In particular, we have $T_{S_i} \neq 0$. Then, there are two cases to consider: 

\medskip

(b1). Suppose that $T_{S_i}$ is a horizontal divisor over $V_i$, that is, ${\rm Supp}(T_{S_i})$ maps surjectively to $V_i$. Now as $T_{S_i}$ is not an exceptional divisor for $q_i: S_i \to S_i''$, there are summands of $B_{S_i''}$ whose coefficients are of the form $\frac{m-1+f+kc_i'}{m}$ with $k \in \Nn$ (see discussion after equation \eqref{eq: coefficients}). Moreover, as $T_{S_i}$ is a horizontal divisor, there is a component of ${q_i}_*(T_{S_i})$ whose restriction to $F''_i$ is non-zero. Thus, at least one coefficients of $B_{F_i''}$ is in $D_{c_i'}(I)$.

\medskip

(b2). Suppose that $T_{S_i}$ is a vertical divisor over $V_i$, that is, the image of ${\rm Supp}(T_{S_i})$ is not equal to $V_i$. Then $(T_{S_i})|_{F_i}=0$. Because $E_{S_i}$ is an exceptional divisor for $q_i$, $(E_{S_i})|_{F_i}$ is $\psi_i$-exceptional. Put them together, we have $\psi_{i*}(R_i|_{S_i})|_{F_i}=0$. By \eqref{eq: fibre1} and the bigness of $M_i''|_{F_i''}$, $M_{F''_i}={\psi_i}_*(M_i|_{F_i})$ is big. 

\medskip

In summary, the above shows that: either $M_{F''_i}$ is big (in particular, non-zero), or at least one coefficient of $B_{F_i''}$ is in $D_{c_i'}(I)$. The generalized polarized pair $({F''_i}, T_{F''_i}+B_{F''_i}+c_i'M_{F''_i})$ satisfies all the assumptions of Proposition \ref{prop: induction of accumulation point} with $\dim F_i''<n$ except that $\{c_i'\}_{i\in\Nn}$ may not be strictly decreasing.

\medskip

Now, if $c=c_i'$ for some $i$, then 
\[
K_{F''_i}+ T_{F''_i}+B_{F''_i}+cM_{F''_i} \equiv 0.
\] Because  $T_{F''_i}$ is a reduced divisor and $1 \in I$, we have 
\[
({F''_i}, T_{F''_i}+B_{F''_i}+cM_{F''_i}) \in \mathfrak{N}_{\dim F''_i }(I,c),
\] thus $c \in N_{\dim F''_i }(I) \subseteq N_{n-1}(I)$. 

\medskip

We can assume that $c \neq c_i'$ for infinitely many $i$. By passing to a subsequence, $c_i'$ is strictly decreasing to $c$, and we obtain the result by applying the induction hypothesis to $({F''_i}, T_{F''_i}+B_{F''_i}+c_i'M_{F''_i})$.

\medskip

\noindent Step 6. In this step, we show the remaining case of Step 5 when $({X'_i}, T'_i+B'_i+c_iM'_i)$ is generalized lc (see the third paragraph of Step 5) and finish the proof of the whole proposition. 

\medskip

Suppose that $({X'_i}, T'_i+B'_i+c_iM'_i)$ is generalized lc for infinite $i$, and $\rho(X_i')=1$ (see the end of Step 3). 

\medskip

When $K_{X'_i}+T'_i+B'_i+cM'_i \equiv 0$, then let $S'_i$ be the normalization of an irreducible component of $T'_i$. Because $\rho(X'_i)=1$, ${M_i'}|_{S'_i } \not\equiv 0$ (recall that we assume $M_{i}'\not\equiv 0$ at the beginning of the proof). By the generalized adjunction on $S'_i$, we have
\[
K_{S_i'}+ \Theta_{S_i'}+cM_{S_i'} \equiv (K_{X'_i}+T'_i+B'_i+cM'_i )|_{S_i'} \equiv 0.
\] After possibly passing to a $\Qq$-factorial model as before, $(S_i',  \Theta_{S_i'}+cM_{S_i'}) \in \mathfrak N_{n-1}(I,c)$. Thus $c \in N_{n-1}(I)$.

\medskip

When $K_{X'_i}+T'_i+B'_i+cM'_i$ is anti-ample, then there exists $c_i''$ such that $K_{X'_i}+T'_i+B'_i+c_i''M'_i \equiv 0$. By $K_{X'_i}+A'_i+B'_i+c_iM'_i \equiv 0$ and $T_i' \geq A_i'$, we see that $c<c_i'' \leq c_i$. Hence, by passing to a subsequence, $\{c_i''\}_{i\in \Nn}$ is strictly decreases to $c$. Then by the generalized adjunction on $S_i'$, we obtain the result by the induction hypothesis just as the case $\lfloor A_i \rfloor \neq 0$ in Step 4.

\medskip

Finally, suppose that $K_{X'_i}+T'_i+B'_i+cM'_i$ is ample. Because $c<c_i$, $K_{X'_i}+A'_i+B'_i+cM'_i$ is anti-ample, and thus $T_i' \neq A_i'$. Since $\rho(X_i')$=1, there exists $r_i \in (0,1)$, such that 
\[
K_{X'_i}+r_iT'_i+(1-r_i)A_i'+B'_i+cM'_i \equiv 0.
\] Because the coefficients of $A_i$ is approaching $1$ and at least one coefficient is not $1$ (because $T_i' \neq A_i'$), the coefficients of $r_iT'_i+(1-r_i)A_i'$ is approaching $1$ and not all of them equal to $1$. Moreover, $({X'_i}, r_iT'_i+(1-r_i)A_i'+B'_i+cM'_i)$ is generalized lc, and by passing to a subsequence, we can assume that the coefficients of $r_iT'_i+(1-r_i)A_i'+B'_i$ lie in a DCC set. But the coefficients must lie in a finite set by the global ACC of generalized polarized pairs (Theorem \ref{thm: global ACC}). This is a contradiction.
\end{proof}

\begin{lemma}\label{lem: ACC for N}
Suppose that $I \subseteq [0,1]$ is a DCC set, then $N_n(I)$ is an ACC set.
\end{lemma}
\begin{proof}
Suppose that there exists a strictly increasing sequence $\{c_i \mid c_i \in N_n(I)\}_{i\in \Nn}$. We claim that the set 
\[
\{\frac{m_i-1+f_i+k_ic_i}{m_i} \leq 1 \mid m_i \in \Nn, k_i \in \Zz_{\geq 0}, f_i \in I_+\}
\] is a DCC set. Otherwise, we can assume that $\{\frac{m_i-1+f_i+k_ic_i}{m_i}\}_{i\in \Nn}$ is a strictly decreasing sequence. By passing to a subsequence, we can assume that $\{m_i\}_{i\in\Nn}, \{f_i\}_{i\in\Nn}$ and $\{k_i\}_{i\in\Nn}$ are non-decreasing sequences as $\Nn, \Zz_{\geq 0}$ and $I_+$ are all DCC sets. We can also assume that $\frac{m_i-1+f_i+k_ic_i}{m_i} <1$, and thus $\{m_{i}\}_{i\in \Nn}$ is bounded. By passing to a subsequence again, we can assume that $\{m_{i}\}_{i\in \Nn}$ is a constant sequence and thus $\{f_i+k_ic_i\}_{i\in\Nn}$ is strictly decreasing. But this leads to a contradiction as $\{f_i\}_{i\in\Nn}$ is non-decreasing and $\{c_i\}_{i\in\Nn}$ is strictly increasing (notice that $k_i$ could be $0$). 

\medskip

The set of varieties in $\cup_{i\in\Nn}\mathfrak{N}(I,{c_i})$ have coefficients in a DCC set, and thus by global ACC (Theorem \ref{thm: global ACC}), they must lie in a finite set. This contradicts to the strictly increasing assumption on $\{c_i\}_{i\in \Nn}$. In fact, either there are infinitely $i$ such that $M'_i \not \equiv 0$ and we are done by $K_{X'_i}+B'_i+c_iM'_i \equiv 0$, or $M'_i \equiv 0$,  with $\frac{m_i-1+f_i+k_ic_i}{m_i}, k_i>0$ as a coefficient of $B'_i$ . In the later case, $\{\frac{m_i-1+f_i+k_ic_i}{m_i}\}_{i \in \Nn}$ is an infinite set. Indeed, by passing to a subsequence, we can assume that $\{f_i+k_ic_i\}_{i\in\Nn}$ is strictly increasing. In particular, they are not $1$ by passing to a tail. Then no matter $\{m_i\}_{i\in\Nn}$ is bounded or unbounded, $\{\frac{m_i-1+f_i+k_ic_i}{m_i}\}_{i \in \Nn}$ is an infinite set. 
\end{proof}

\begin{corollary}\label{cor: accumulation point}
Let $I$ be a DCC set such that $I=I_+$. Assume that $1\in I$ with $1$ the only possible accumulation point, then for any $1 \leq k \leq n$, $\lim^k N_n(I) \subseteq N_{n-k}(I)$.
\end{corollary}
\begin{proof}
By induction, it is enough to show that $\lim^1 N_n(I) \subseteq N_{n-1}(I)$. By Lemma \ref{lem: ACC for N}, if $\{c_i\}_{i\in\Nn}$ has an accumulation point $c$, then there is a strictly increasing sequence converging to $c$. Moreover, as $0 \in N_{n-1}(I)$, we can assume that $c > 0$. Then the claim follows from Proposition \ref{prop: induction of accumulation point} by taking $A_i=0$.
\end{proof}

\begin{proof}[Proof of Theorem \ref{thm: accumulation point of pet}]
This follows from Proposition \ref{prop: limit in numerically trivial}, Corollary \ref{cor: accumulation point} and Lemma \ref{le: picard number 1}.
\end{proof}

\subsection{An application}\label{subsection: applications}

We demonstrate an application of Theorem \ref{thm: accumulation point of pet} towards Fujita's spectrum conjecture for large $k$. In practice, as long as $I$ is DCC with $1$ to be the only possible accumulation point, we can always enlarge $I$ so that it satisfies all the assumptions in Theorem \ref{thm: accumulation point of pet}. Notice that this could only enlarge the accumulation points.

\begin{lemma}\label{lem: coefficient set}
If $I \subseteq [0,1]$ is a DCC set with $1$ to be the only possible accumulation point, then $J \coloneqq (I\cup \{1\})_+$ is a DCC set such that $J_+=J$ with $1 \in J$ to be the only possible accumulation point.
\end{lemma}
\begin{proof}
If $I$ is DCC, then $J$ is DCC and $J_+=J$ by definition. It is enough to show that $1$ is
also the only possible accumulation point of $J$. Otherwise, there exists a sequence $\{c_+^{i}\}_{i\in \Nn}$ of $J$ approaching $c_+<1$. Each $c_+^i = \sum_{j}^{n_i} a_{ij}$, where $a_{ij} \in I$ (repetition is allowed). We claim that $n_i$ is bounded above. Otherwise there exists a subsequence sequence $\{a_{k_i j_{k_i}}\}$ decreasing to $0$ which contradicts to that $I$ is DCC. By passing to a subsequence, we can assume that $n_i=n$ is a fixed number. For each $c_+^i$, there is an $n$-tuple $(a_{i1}, \ldots, a_{in})$ (the order does not matter). By passing to a subsequence again, we can assume that for each $k$, $\{a_{ik}\}_{i\in \Nn}$ is an increasing sequence. Hence there exists $\lim_i a_{ik} = a_k$, and $0<a_k<1$. This is a contradiction.
\end{proof}

\begin{proof}[Proof of Proposition \ref{prop: Fujita's conjecture for n-1}]
By Lemma \ref{lem: coefficient set}, replacing $I$ by $(I\cup\{1\})_+$, we can assume that $I$ satisfies all the assumptions of Theorem \ref{thm: accumulation point of pet}. 

\medskip

For $k=n-1$, $\lim^{n-1} (\PET_n(I)) \subseteq K_{1}(I)$ by Theorem \ref{thm: accumulation point of pet}. It suffices to give an upper bound for $K_{1}(I)$. Let $c \in K_{1}(I)$, by definition in Section \ref{sec: accumulation points}, there exists a smooth curve $X'$, and $B', M'$ such that $K_{X'}+B'+cM' \equiv 0$. There are two cases to consider. If $M' \equiv 0$, then some coefficient of $B'$ lies in $D_c(I)$. This coefficient is of the form 
\begin{equation}\label{eq: coeff of B}
\frac{m-1+f+kc}{m}, \quad m,k\in \Nn \text{~and~} f\in I_+.
\end{equation} By generalized klt assumption, all the coefficients of $B'$ are less than $1$, hence $kc<1$, and thus $c<1$. If $M' \not\equiv 0$, then $M'$ is an ample Cartier divisor, hence $c \leq 2$.

\medskip

For $k=n$, by Theorem \ref{thm: accumulation point of pet} and \eqref{eq: N_0}, 
\[
\lim{}^n(\PET_n(I))\subseteq K_0(I)=\{\frac{1-f}{m} \geq 0 \mid f\in I_+, m\in \Nn \}. 
\] Thus $c \leq 1$ for any $c\in \lim^n(\PET_n(I))$.
\end{proof}

Proposition \ref{prop: Fujita's conjecture for n-1} gives upper bounds for the first and the second accumulation points of surfaces and they are sharp under our conditions.

\begin{remark}
Corollary \ref{cor: accumulation point} shows that a $k$-th accumulation point $c$ lies in $N_{n-k}(I)$. When $n-k>1$, a similar argument as Proposition \ref{prop: Fujita's conjecture for n-1} has the following difficulty: when $M' \not\equiv 0$, $M'$ is known to be an ample Weil divisor which may not necessarily be Cartier. Thus, there is no upper bound for $c$. If one works with $M$ instead of $M'$, then one gains the nef and Cartier property but loses the bigness. However, we are able to overcome such difficulty when $M$  is assumed to be ample and Cartier in $\PET_n(I)$ \eqref{eq: PET_n(I)} (see \cite{Li18}).
\end{remark}

\bibliographystyle{alpha}

\bibliography{bibfile}

\newcommand{\etalchar}[1]{$^{#1}$}
\begin{thebibliography}{BCHM10}

\bibitem[BCHM10]{BCHM10}
Caucher Birkar, Paolo Cascini, Christopher~D. Hacon, and James McKernan.
\newblock Existence of minimal models for varieties of log general type.
\newblock {\em J. Amer. Math. Soc.}, 23(2):405--468, 2010.

\bibitem[Bir16]{Bir16b}
Caucher Birkar.
\newblock Singularities of linear systems and boundedness of fano varieties.
\newblock {\em arXiv:1609.05543}, 2016.

\bibitem[BZ16]{BZ16}
Caucher Birkar and De-Qi Zhang.
\newblock Effectivity of {I}itaka fibrations and pluricanonical systems of
  polarized pairs.
\newblock {\em Publ. Math. Inst. Hautes \'Etudes Sci.}, 123:283--331, 2016.

\bibitem[DC16]{DiC16}
Gabriele Di~Cerbo.
\newblock On {F}ujita's log spectrum conjecture.
\newblock {\em Math. Ann.}, 366(1-2):447--457, 2016.

\bibitem[DC17]{DiC17}
Gabriele Di~Cerbo.
\newblock On {F}ujita's spectrum conjecture.
\newblock {\em Adv. Math.}, 311:238--248, 2017.

\bibitem[Fuj92]{Fuj92}
Takao Fujita.
\newblock On {K}odaira energy and adjoint reduction of polarized manifolds.
\newblock {\em Manuscripta Math.}, 76(1):59--84, 1992.

\bibitem[Fuj96]{Fuj96}
Takao Fujita.
\newblock On {K}odaira energy of polarized log varieties.
\newblock {\em J. Math. Soc. Japan}, 48(1):1--12, 1996.

\bibitem[HL17]{HL17}
Jingjun Han and Zhan Li.
\newblock On {F}ujita's conjecture for pseudo-effective thresholds.
\newblock {\em arXiv:1705.08862, to appear in Math. Res. Lett.}, 2017.

\bibitem[HL18]{HL18}
Jingjun Han and Zhan Li.
\newblock Weak {Z}ariski decompositions and log terminal models for generalized
  polarized pairs.
\newblock {\em arXiv:1806.01234}, 2018.

\bibitem[HMX14]{HMX14}
Christopher~D. Hacon, James McKernan, and Chenyang Xu.
\newblock A{CC} for log canonical thresholds.
\newblock {\em Ann. of Math. (2)}, 180(2):523--571, 2014.

\bibitem[K{\etalchar{+}}92]{Fli92}
J\'anos Koll\'ar et~al.
\newblock {\em Flips and abundance for algebraic threefolds}.
\newblock Soci\'et\'e Math\'ematique de France, Paris, 1992.

\bibitem[Li18]{Li18}
Zhan Li.
\newblock Fujita's conjecture on iterated accumulation points of
  pseudo-effective thresholds.
\newblock {\em arXiv:1812.04262}, 2018.

\bibitem[MP04]{MP04}
James McKernan and Yuri Prokhorov.
\newblock Threefold thresholds.
\newblock {\em Manuscripta Math.}, 114(3):281--304, 2004.

\bibitem[Sho88]{Sho88}
Vyacheslav~V. Shokurov.
\newblock Problems about {F}ano varieties.
\newblock {\em Birational Geometry of Algebraic Varieties, Open Problems. The
  XXIIIrd International Symposium, Division of Mathematics, The Taniguchi
  Foundation}, pages 30--32, 1988.

\bibitem[Sho92]{Sho92}
Vyacheslav~V. Shokurov.
\newblock Three-dimensional log perestroikas.
\newblock {\em Izv. Ross. Akad. Nauk Ser. Mat.}, 56(1):105--203, 1992.

\end{thebibliography}

\end{document}